\documentclass[11pt]{article}
\usepackage{amsmath,amssymb, amscd}

\newtheorem{propo}{{\bf Proposition}}[section]
\newtheorem{coro}[propo]{{\bf Corollary}}
\newtheorem{lemma}[propo]{{\bf Lemma}} \newtheorem{theor}[propo]{{\bf
Theorem}} 

\newenvironment{proof}{{\bf Proof.}}{$\Box$}

\begin{document}

\vspace*{1.0in}

\begin{center} ON MAXIMAL SUBALGEBRAS AND A GENERALISED JORDAN-H\"OLDER THEOREM FOR LIE ALGEBRAS
\end{center}
\bigskip

\begin{center} DAVID A. TOWERS 
\end{center}
\bigskip

\begin{center} Department of Mathematics and Statistics

Lancaster University

Lancaster LA1 4YF

England

d.towers@lancaster.ac.uk 
\end{center}
\bigskip

\begin{abstract}
The purpose of this paper is to continue the study of chief factors of a Lie algebra and to prove a further strengthening of the Jordan-H\"older Theorem for chief series.
\end{abstract}

\section{Introduction}
Throughout $L$ will denote a finite-dimensional Lie algebra over a field $F$. The factor algebra $A/B$ is called a {\em chief factor} of $L$ if $B$ is an ideal of $L$ and $A/B$ is a minimal ideal of $L/B$. A chief factor $A/B$ is called {\em Frattini }if $A/B\subseteq \phi
\left( L/B\right) .$ This concept was first introduced in \cite{prefrat}.
\par

If there is a subalgebra, $M$ such that $L=A+M$ and $B\subseteq A\cap M,$ we
say that $A/B\ $is$\ $a \textit{supplemented} chief factor of $L,$ and that $%
M$ is a \textit{supplement} of $A/B\ $in $L.$ Also, if $A/B$ is a
non-Frattini chief factor of $L$, then $A/B$ is \textit{supplemented} by a
maximal subalgebra $M$ of $L.$
\par

If $A/B\ $is$\ $a chief factor of $L$ supplemented by a subalgebra $M$ of $%
L, $ and \linebreak $A\cap M=B$ then we say that $A/B$ is \textit{%
complemented} chief factor of $L,$ and $M$ is a \textit{complement} of $A/B$
in $L$. When $L$ is \textit{solvable}, it is easy to see that a chief factor
is Frattini  if and only if it is not complemented.
\par

If $U$ is a subalgebra of $L,$ the \textit{core} of $U$,  $U_{L}$,
is the largest ideal of $L$ contained in $U$. We say
that $U$ is core-free in $L$ if $U_{L}=0.$ 

We shall call $L$ \textit{%
primitive} if it has a core-free maximal subalgebra. Then we have the following characterisation of primitive Lie algebras.

\begin{theor}\label{t:1} (\cite[Theorem 1.1]{prim})
\begin{itemize}
\item[(i)]  A Lie algebra $L$ is primitive if and only if there exists a subalgebra $M$ of $L$ such
that $L = M+A$ for all minimal ideals $A$ of $L$.
\item[(ii)]  Let $L$ be a primitive Lie algebra. Assume that $U$ is a core-free maximal subalgebra
of $L$ and that $A$ is a non-trivial ideal of $L$. Write $C = C_L(A)$.
Then $C \cap U = 0$. Moreover, either $C = 0$ or $C$ is a minimal ideal of $L$.
\item[(iii)]  If L is a primitive Lie algebra and $U$ is a core-free maximal subalgebra of $L$,
then exactly one of the following statements holds:
\begin{itemize}
\item[(a)] $Soc(L) = A$ is a self-centralising abelian minimal ideal of $L$ which is complemented by $U$; that is, $L = U \dot{+} A$.
\item[(b)] $Soc(L) = A$ is a non-abelian minimal ideal of $L$ which is
supplemented by $U$; that is $L = U+A$. In this case $C_L(A) = 0$.
\item[(c)] $Soc(L) = A \oplus B$, where $A$ and $B$ are the two unique minimal ideals of $L$ and both are complemented by $U$; that is, $L = A \dot{+} U = B \dot{+} U$. In this case $A = C_L(B)$, $B = C_L(A)$, and $A$, $B$ and $(A+B) \cap U$ are nonabelian isomorphic algebras.
\end{itemize}
\end{itemize}
\end{theor}

We say that $L$ is  
\begin{itemize}
\item {\em primitive of type $1$} if it has a unique minimal ideal that is abelian;
\item {\em primitive of type $2$} if it has a unique minimal ideal that is  non-abelian; and
\item {\em primitive of type $3$} if it has precisely two distinct minimal ideals each of which is  non-abelian.
\end{itemize}
\bigskip

If $A/B$ be a supplemented chief factor of $L$ for which $M$ is a maximal subalgebra of $L$ supplementing $A/B$ in $L$ such that $L/M_L$ is monolithic and primitive, we call $M$ a {\em monolithic maximal subalgebra} supplementing $A/B$ in $L$.  Note that \cite[Proposition 2.5 (iii) and (iv)]{prim} show that such an $M$ exists.
\par

We say that two chief factors are {\em $L$-isomorphic}, denoted by `$\cong_L$', if they are isomorphic both as algebras and as $L$-modules. Note that if $L$ is a primitive Lie algebra of type $3$, its two minimal ideals are not $L$-isomorphic, so we introduce the following concept. We say that two chief factors of $L$ are {\em $L$-connected} if either they are $L$-isomorphic, or there exists an epimorphic image $\overline{L}$ of $L$ which is primitive of type $3$ and whose minimal ideals are $L$-isomorphic, respectively, to the given factors. (It is clear that, if two chief factors of $L$ are $L$-connected and are not $L$-isomorphic, then they are nonabelian and there is a single epimorphic image of $L$ which is primitive of type $3$ and which connects them.)

Our primary objective is to generalise further the version of the Jordan-H\"{o}lder Theorem for chief series of $L$ established in \cite{prim}. Our result could probably be obtained from \cite{laf}, but we prefer to follow the approach adopted for groups (though in a more general context) in \cite{b-e} as interesting results and concepts are obtained on the way. 

\section{Preliminary results on chief factors}
Let $A/B$ and $C/D$ be chief factors of $L$. We write $A/B \searrow C/D$ if $A=B+C$ and $B \cap C=D$. Clearly, if $A/B$ is abelian, then so is $C/D$.
\begin{lemma}\label{l:1} Let $A/B \searrow C/D$. Then
\begin{itemize}
\item[(i)] if $A/B$ is supplemented by $M$ in $L$, then so is $C/D$;
\item[(ii)] if $M$ supplements $A/B$ in $L$ and $K$ supplements $B/D$ in $L$, then $C+M\cap K$ supplements $A/C$ in $L$ and, in this case, $M\cap K$ supplements $A/D$ in $L$; and
\item[(iii)] (i) and (ii) both hold with `supplemented' replaced by `complemented'.
\end{itemize}
If, further, $C/D$ is non-abelian, then
\begin{itemize}
\item[(iv)] the set of monolithic supplements of $C/D$ in $L$ coincides with the set of monolithic supplements of $A/B$ in $L$;
\item[(v)] if $B/D$ is an abelian chief factor of $L$ then the (possibly empty) set of complements of $B/D$ in $L$ coincides with the set of complements of $A/C$ in $L$.
\end{itemize}
\end{lemma}
\begin{proof} We have $A=B+C$ and $B \cap C=D$. 
\begin{itemize}
\item[(i)] Suppose that $L=A+M$ and $B \subseteq A\cap M$. Then $L=B+C+M=C+M$ and $D=B\cap C \subseteq A\cap M\cap C=M\cap C$.
\item[(ii)] Suppose that $L=A+M=B+K$, $B \subseteq A\cap M$ and $D \subseteq B\cap K$. Then $A+C+M\cap K=A+M\cap K=B+C+M\cap K=C+M\cap(K+B)=C+M=L$. Furthermore, $C=C+D \subseteq C+B\cap K \subseteq C+(A\cap M\cap K)=A\cap (C+M\cap K)$, so $C+M\cap K$ supplements $A/C$ in $L$.
\par 

Moreover, in this case, $L=A+M\cap K$ and $D \subseteq B\cap K \subseteq A\cap M\cap K$, so $M\cap K$ supplements $A/D$ in $L$.
\item[(iii)] Simply substituting equalities for the inequalities in the above proofs will yield the corresponding results for complements.
\end{itemize}
For the remaining parts we can assume, without loss of generality that $D=0$ and that $C$ is a non-abelian minimal ideal of $L$. It follows that $[C,B] \subseteq C\cap B=0$.
\begin{itemize}
\item[(iv)] If $M$ is a monolithic supplement of $C$ in $L$ then $L=C+M$ and $L/M_L$ is a primitive Lie algebra of type $2$. Now Soc$(L/M_L)=(C+M_L)/M_L$ and $C_L(C)=C_L((C+M_L)/M_L)=M_L$. Hence $B\subseteq C_L(C)=M_L\subseteq M$. Thus $L=A+M$ with $B\subseteq A\cap M$ and $M$ is a monolithic supplement of $A/B$ in $L$.
\par

Conversely, if $M$ is a monolithic supplement of $A/B$ in $L$, then $M$ supplements $C$ in $L$, by (i). 
\item[(v)] Suppose that $B$ is an abelian ideal of $L$ complemented by $M$, so $L=B\dot{+} M$. Then $C_L(B)=B\oplus M_L$ and $A=B\oplus A\cap M_L$. Since $C$ is non-abelian, this implies that $C=C^2=A^2\subseteq A\cap M_L$. Thus $C\subseteq M$ and $M$ complements $A/C$ in $L$. 
\par

Conversely, suppose that $L=A+M$ and $A\cap M=C$. Then $L=B+C+M=B+M$ and $C=A\cap M=C+B\cap M$, so $B\cap M\subseteq B\cap C=0$. Hence $M$ complements $B$ in $L$.
\end{itemize}
\end{proof}

\begin{lemma}\label{l:2} Let $U$ and $S$ be two maximal subalgebras of a Lie algebra $L$ such that $U_L \neq S_L$. Suppose that $U$ and $S$ supplement the same chief factor $A/B$ of $L$. Then $M=A+U\cap S$ is a maximal subalgebra of $L$ such that $M_L=A+U_L\cap S_L$.
\begin{itemize}
\item[(i)] Assume that $A/B$ is abelian. Then $M$ is a maximal subalgebra of type $1$ which complements the chief factors $U_L/U_L\cap S_L$ and $S_L/U_L\cap S_L$. Moreover, $M\cap U= M\cap S=U\cap S$.
\item[(ii)] Assume that $A/B$ is non-abelian. Then either $U$ or $S$ is of type $3$. Suppose that $U$ is of type $3$ and $S$ is monolithic. Then $U_L \subset S_L=C_L(A/B)$. Moreover, $M$ is a maximal subalgebra of type $2$ of $L$ which supplements the chief factor $S_L/U_L$.
\item[(iii)] Assume that $U$ and $S$ are of type $3$. Then $M$ is a maximal subalgebra of type $3$ of $L$ which complements the chief factors $(A+S_L)/M_L$ and $(A+U_L)/M_L$. Moreover $M\cap U=M\cap S=U\cap S$.
\end{itemize}
\end{lemma}
\begin{proof} We have that $L=A+U=A+S$ and $B\subseteq A\cap U\cap S$.
\begin{itemize}
\item[(i)] Let $A/B$ be abelian and put $C=C_L(A/B)$. First note that $B \subseteq A\cap U \subset A$ and $A\cap U$ is an ideal of $L$ since $A/B$ is abelian, so $B=A\cap U$. Thus $M\cap U=(A+U\cap S)\cap U=A\cap U+U\cap S=B+U\cap S=U\cap S$. Similarly $M\cap S=U\cap S$. Since $U \neq S$, $M$ is a proper subalgebra of $L$. 
\par

Clearly $C=A+U_L=A+S_L$. But $B \subseteq A\cap (U_L+S_L) \subseteq A$, so $A\cap (U_L+S_L)= B$ or $A$. The former implies that $U_L+S_L=U_L+A\cap (U_L+S_L)=U_L+B=U_L$ and $U_L+S_L=S_L$ similarly, contradicting the fact that $U_L \neq S_L$. Hence the latter holds and $C=U_L+S_L$. But now $U_L/U_L\cap S_L \cong_L C/S_L \cong_L A/B$ since $A\cap S_L=B$.
\par

We have $L=S+A=S+C=S+U_L+S_L=S+U_L$, so $M+U_L=A+U\cap S+U_L=A+U\cap(S+U_L)=A+U=L$. Moreover, $U_L\cap S_L \subseteq M\cap U_L$ so $M$ is a maximal subalgebra of $L$ which complements the abelian chief factor $U_L/U_L\cap S_L$. Similarly $M$ complements $S_L/U_L\cap S_L$.
\par

Finally $L=M+U_L=M+C$ and $M\cap C=M\cap (A+U_L)=A+M\cap U_L=A+ U_L\cap S_L$, so $M$ also complements $C/(A+U_L\cap S_L)$, whence $M_L=A+U_L\cap S_L$.
\item[(ii)] Assume that $A/B$ is non-abelian. If $U$ and $S$ were both monolithic of type $2$, then $U_L=S_L=C_L(A/B)$, contradicting our hypothesis. It follows that either $U$ or $S$ is of type $3$.
\par

So assume that $U$ is of type $3$ and that $S$ is monolithic. Then $S_L=C_L(A/B)$. Note that $(A+U_L)/U_L$ is a chief factor of $L$ which is $L$-isomorphic to $A/B$. Hence $(A+U_L)/U_L$ and $S_L/U_L$ are the two minimal ideals of the primitive Lie algebra $L/U_L$ of type $3$. Both of these are complemented by $U$; in particular, $L=U+S_L$.
\par

Now $M+S_L=A+U\cap S+S_L=A+(U+S_L)\cap S=A+S=L$ and $U_L=U_L+B=U_L+A\cap S_L=(U_L+A)\cap S_L \subseteq (U\cap S+A)\cap S_L=M\cap S_L$. It follows that $M$ supplements the chief factor $S_L/U_L$ in $L$.
\par

The quotient algebra 
\[ \frac{L}{A+U_L}=\frac{A+S_L}{A+U_L} + \frac{M}{A+U_L}
\] is primitive of type $2$, by \cite[Theorem 1.7, part 2]{prim}. But then the ideal $M_L/(A+U_L)$ must be trivial, since otherwise we have $A+S_L \subseteq M_L$ which implies that $S_L \subseteq M$, a contradiction. Hence $M_L=A+U_L$.
\par

Let $T$ be a subalgebra of $L$ such that $U\cap S \subseteq T \subseteq U$. Then $S=U\cap S+S_L \subseteq T+S_L \subseteq U+S_L=L$. Since $S$ is a maximal subalgebra of $L$, we have that $T+S_L=S$ or $L$. But then, since $T\cap S_L=U\cap S_L=U_L$, we have $U\cap (T+S_L)=T+U\cap S_L=T$, so $T=U\cap S$ or $T=U\cap L=U$. Hence $U\cap S$ is a maximal subalgebra of $U$. The image of $U\cap S/U\cap A$ under the isomorphism from $U/U\cap A$ onto $L/A$ is $M/A$, and so $M$ is a maximal subalgebra of $L$ of type $2$.
\item[(iii)] Assume now that $U$ and $S$ are maximal subalgebras of type $3$, so that the quotient algebras $L/U_L$ and $L/S_L$ are primitive Lie algebras of type $3$.
\par

If $C=C_L(A/B)$, then $U$ complements the chief factors $(A+U_L)//U_L$ and $C/U_L$. Similarly,  $S$ complements the chief factors $(A+S_L)//S_L$ and $C/S_L$. In particular, $U_L \not \subseteq S_L$ and $S_L \not \subseteq U_L$. Hence $L=U+S_L=S+U_L$. But now, by a similar argument to that at the end of (ii), we have that $M=A+U\cap S$ is a maximal subalgebra of $L$.
\par

Now $C/S_L$ and $C/U_L$ are chief factors of $L$ and $U_L \neq S_L$, so $C=U_L+S_L$. Put $H=U_L\cap S_L$. Then
\[ \frac{A+U_L}{A+H} \cong_L \frac{U_L}{U_L\cap (A+H)}= \frac{U_L}{H} \cong_L \frac{C}{S_L},
\] and so $(A+U_L)/(A+H)$ is a chief factor of $L$ and 
\[C_L\left( \frac{A+U_L}{A+H}\right)=C_L\left(\frac{C}{S_L}\right)=A+S_L.
\] Similarly  $(A+S_L)/(A+H)$ is a chief factor of $L$ and 
\[ C_L\left(\frac{A+S_L}{A+H}\right)=A+U_L.
\] It follows that the quotient Lie algebra $\bar{L}=L/(A+H)$ has two minimal ideals, namely $\bar{N}=(A+S_L)/(A+H)$ and $C_{\bar{L}}(\bar{N})=(A+U_L)/(A+H)$. But $A+M+S_L=A+U\cap S+S_L=A+((U+S_L)\cap S)=A+S=L$. Since $U$ complements $C/U_L$, we have that $U\cap (U_L+S_L)=U_L$, so $U\cap S_L=U_L\cap S_L=H$ and $M\cap (A+S_L)=A+(U\cap S\cap (A+S_L))=A+U\cap S_L=A+H$. Similarly $L=A+M+U_L$ and $M\cap (A+U_L)=A+H$. It follows that the maximal subalgebra $\bar{M}=M/(A+H)$ of $\bar{L}$ complements $\bar{N}$ and $C_{\bar{L}}(\bar{N})$, and thus $\bar{L}$ is a primitive Lie algebra of type $3$. Hence $M_L=A+H$.
\par

Finally note that $M \cap U= (A+U\cap S)\cap U= A\cap U+U\cap S= B+U\cap S=U\cap S$. Similarly $M\cap S=U\cap S$.
\end{itemize}
\end{proof}
\bigskip

The following result is straightforward to check.

\begin{theor} Suppose that $L=B+U$, where $B$ is an ideal of $L$ and $U$ is a subalgebra of $L$. Then $L/B \cong U/B\cap U$ and the following hold.
\begin{itemize}
\item[(i)] If 
\begin{equation} B=B_n < \ldots < B_0=L
\end{equation}
is part of a chief series of $L$, then
\begin{equation} B\cap U = B_n\cap U < \ldots <B_0\cap U=U
\end{equation} is part of a chief series of $U$. If $M$ is a maximal subalgebra of $L$ which supplements a chief factor $B_i/B_{i+1}$ in (1), then $M\cap U$ is a maximal subalgebra of $U$ which supplements the chief factor $B_i\cap U/B_{i+1}\cap U$ in (2). Moreover, $(M\cap U)_U=M_L\cap U$.
\item[(ii)] Conversely, if
\begin{equation} B\cap U=U_n< \ldots <U_0=U
\end{equation} is part of a chief series of $U$, then
\begin{equation} B=B+U_n< \ldots < B+U_0=B+U=L
\end{equation} is part of a chief series of $L$. If $T$ is a maximal subalgebra of $U$ which supplements a chief factor $U_i/U_{i+1}$ in (3), then $B+T$ is a maximal subalgebra of $L$ which supplements the chief factor $(B+U_i)/(B+U_{i+1})$ in (4). Moreover, $(B+T)_L=B+T_U$. 
\end{itemize}
\end{theor}

\begin{lemma}\label{l:3} Let $K$ and $H$ be ideals of a Lie algebra $L$ and let
\[ K=Y_0 \subset Y_1 \subset \ldots \subset Y_{m-1} \subset Y_m=H
\] be part of a chief series of $L$ between $K$ and $H$. Suppose that $A/B$ is a chief factor of $L$ between $K$ and $H$. Then
\begin{itemize}
\item[(i)] if $A+Y_j=B+Y_j$, then $A+Y_k=B+Y_k$ for $j \leq k \leq m$;
\item[(ii)] if $A\cap Y_{j-1}=B\cap Y_{j-1}$, then  $A\cap Y_{k-1}=B\cap Y_{k-1}$ for $1 \leq k \leq j$;
\item[(iii)] if $ B+Y_{j-1} \subset A+Y_{j-1}$, then $B+Y_{k-1} \subset A+Y_{k-1}$ for $1 \leq k \leq j$ and $A\cap Y_{j-1}=B\cap Y_{j-1}$. In this case,
\[ \frac{A+Y_{j-1}}{B+Y_{j-1}} \searrow \frac{A+Y_{k-1}}{B+Y_{k-1}} \searrow \frac{A}{B}.
\]
\item[(iv)] If $B\cap Y_j \subset A\cap Y_j$, then $B\cap Y_k \subset A\cap Y_k$ for $j \leq k \leq m$ and $A+Y_j=B+Y_j$. Moreover,
\[ \frac{A}{B} \searrow \frac{A\cap Y_k}{B\cap Y_k} \searrow \frac{A\cap Y_j}{B\cap Y_j}.
\]
\end{itemize}
\end{lemma}
\begin{proof} \begin{itemize}
\item[(i)] This is clear.
\item[(ii)] This is just the dual of (i).
\item[(iii)] The first assertion follows from (i). Now 
\[ (A+Y_{k-1})+(B+Y_{j-1})=A+Y_{j-1} \hbox{ and } A+Y_{k-1}=B+(A+Y_{k-1}).
\] Moreover, $B \subseteq B+A\cap Y_{j-1}=A\cap (B+Y_{j-1}) \subseteq A$. Since $A/B$ is a chief factor of $L$, we have either $B=B+A\cap Y_{j-1}=A\cap (B+Y_{j-1})$ or $A\cap (B+Y_{j-1})=A$. If the latter holds then $A \subseteq B+Y_{j-1}$, which implies that $A+Y_{j-1}=B+Y_{j-1}$, a contradiction. Hence $A\cap Y_{j-1} \subseteq B$, and so $A\cap Y_{j-1}=B\cap Y_{j-1}$. But now $A\cap Y_{k-1}=B\cap Y_{k-1}$, by (ii). Thus
\[ A\cap (B+Y_{k-1})=B+A\cap Y_{k-1}=B+B\cap Y_{k-1}=B,
\] and
\begin{align}
(B+Y_{j-1})\cap (A+Y_{k-1}) & = (B+Y_{j-1})\cap A+Y_{k-1} \nonumber \\
                                                & =B+Y_{j-1}\cap A+Y_{k-1} \nonumber \\
                                                &=B+B\cap Y_{j-1}+Y_{k-1}=B+Y_{k-1}, \nonumber
\end{align}
which completes the proof.
\item[(iv)] This is the dual of (iii).
\end{itemize}
\end{proof}

Let $A/B$ and $C/D$ be chief factors of $L$ such that $A/B \searrow C/D$. If $A/B$ is a Frattini chief factor and $C/D$ is supplemented by a maximal subalgebra of $L$, then we call this situation an {\em m-crossing}, and denote it by $[A/B \searrow C/D]$.
\par

Note that if $[A/B \searrow C/D]$ is an $m$-crossing then $C/D$ must be abelian. For, if $C/D$ is a supplemented nonabelian chief factor, then it has a monolithic supplement, by \cite[Proposition 2.5]{prim}, and so $A/B$ must also be supplemented, by Lemma \ref{l:1} (iv).

\begin{theor}\label{t:1} Let $A/C$, $C/D$ and $B/D$ be chief factors of $L$. If $[A/B \searrow C/D]$ is an $m$-crossing, then so is $[A/C \searrow B/D]$. Moreover, in this case a maximal subalgebra $M$ supplements $C/D$ if and only if $M$ supplements $B/D$.
\end{theor}
\begin{proof} Without loss of generality we can assume that $D=0$. Suppose that $B$ and $C$ are minimal ideals of $L$, $A/B$ is a Frattini chief factor and $C$ is supplemented by a maximal subalgebra $M$ of $L$. Then we show that $A/C$ is a Frattini chief factor of $L$ and $B$ is supplemented by $M$.
\par

If $B \subseteq M$, then $L=A+M$ and $B \subseteq A\cap M$, so $M$ supplements $A/B$ in $L$, a contradiction. Hence $B \not \subseteq M$ and $M$ supplements $B$.
\par

Suppose that $K$ is a maximal subalgebra of $L$ that supplements $A/C$ in $L$, so $L=A+K$ and $C \subseteq A\cap K$. Then $L=A+K=B+C+K=B+K$, so $K$ also supplements $B$ in $L$. Since $C \not \subseteq M_L$ and $C \subseteq K_L$, there is a maximal subalgebra $J=B+M\cap K$ such that $J_L=B+M_L\cap K_L$, by Lemma \ref{l:2}. If $A \subseteq J$ then $A=A\cap J_L=B+M_L\cap K_L\cap A= B+M_L\cap C=B$, which is a contradiction. Hence $J$ supplements $A/B$. But $A/B$ is a Frattini chief factor of $L$, so this is not possible. It follows that $A/C$ is a Frattini chief factor of $L$.
\end{proof}

\begin{propo}\label{p:1} With the same hypotheses as in Lemma \ref{l:3} assume that $A/B$ is a supplemented chief factor of $L$. Let 
\[ j'= \hbox{max} \{j : (A+Y_{j-1})/(B+Y_{j-1}) \hbox{ is a supplemented chief factor of } L \}
\] and put $X=Y_{j'}$ and $Y=Y_{j'-1}$. Then $X/Y$ is a supplemented chief factor in $L$. Furthermore the following conditions are satisfied.
\begin{itemize}
\item[(i)] If $A+X=B+X$, then $A+X=A+Y$ and 
\[ \frac{A}{B} \swarrow \frac{A+X}{B+Y} \searrow \frac{X}{Y}. 
\]
Moreover, $A\cap Y=B\cap Y=B\cap X$ and 
\[ \frac{A}{B} \searrow \frac{A\cap X}{B\cap Y} \swarrow \frac{X}{Y}.
\]
\item[(ii)] If $A+X \neq B+X$ then 
\[ \left[\frac{A+X}{B+X} \searrow \frac{A+Y}{B+Y} \right]
\]
 is an $m$-crossing and 
\[\frac{A}{B} \swarrow \frac{A+Y}{B+Y} \hbox{ and } \frac{B+X}{B+Y} \searrow \frac{X}{Y}. 
\]
\end{itemize}
In particular, in both cases, $(A+Y)/(B+Y)$ and $(B+X)/(B+Y)$ are supplemented chief factors of $L$.
\end{propo}
\begin{proof} Note first that $(A+Y_0)/(B+Y_0)=A/B$ is a supplemented chief factor of $L$, so $j'$ is well-defined.
\par

Suppose that $B+X=B+Y$. Then $A+X=A+Y$, so $(A+X)/(B+X)=(A+Y)/(B+Y)$ is supplemented, contradicting the choice of $j'$. Hence $(B+X)/(B+Y) \searrow X/Y$ and  $(B+X)/(B+Y)$ is a chief factor of $L$.
\begin{itemize}
\item[(i)] Suppose that $A+X=B+X$. Then $B+Y \subset A+Y \subseteq A+X=B+X$, so $A+X=A+Y$. Also $A/B \swarrow (A+X)/(B+Y) \searrow X/Y$ by Lemma \ref{l:3} (iii). 
\par

Moreover, $A=A\cap (B+X)=B+A\cap X$, so $A/B \searrow A\cap X/B\cap X$. But now $A\cap Y=B\cap Y=B\cap X$, by Lemma \ref{l:3} (iii). Hence $A/B \searrow A\cap X/B\cap Y \swarrow X/Y$.
\par

In this case 
\[ \frac{B+X}{B+Y}=\frac{A+X}{B+Y}=\frac{A+Y}{B+Y}
\] is supplemented, by the definition of $j'$.
\item[(ii)] Suppose now that  $A+X\neq B+X$. Then $(A+X)/(B+X)$ is a Frattini chief factor of $L$, by the choice of $j'$. Now $B+Y \subseteq (B+X)\cap (A+Y) \subseteq A+Y$. If $(B+X)\cap (A+Y)=A+Y$, then $A+Y \subseteq B+X$ so $A+X=B+X$, a contradiction. Hence $B+Y=(B+X)\cap (A+Y)$ and $[(A+X)/(B+X) \searrow (A+Y)/(B+Y)]$ is an $m$-crossing. Moreover, $A\cap Y=B\cap Y= B\cap X$, by Lemma \ref{l:3} (iii) again, and so we have $A/B \swarrow (A+Y)/(B+Y)$ and $(B+X)/(B+Y) \searrow X/Y$.
\par

Since  $[(A+X)/(B+X) \searrow (A+Y)/(B+Y)]$ is an $m$-crossing, it follows from Theorem \ref{t:1} that $(A+Y)/(B+Y)$ and $(B+X)/(B+Y)$ are supplemented chief factors of $L$.
\end{itemize}
In either case, $(B+X)/(B+Y)$ is a supplemented chief factor of $L$ and $(B+X)/(B+Y) \searrow X/Y$, so $X/Y$ is a supplemented chief factor of $L$.
\end{proof}

\begin{propo}\label{p:2}  With the same hypotheses as in Lemma \ref{l:3} assume that $A/B$ is a Frattini chief factor of $L$. Let
 \[ j'= \hbox{max} \{j : A\cap Y_j/B\cap Y_j \hbox{ is a Frattini chief factor of } L \}
\] and put $X=Y_{j'}$ and $Y=Y_{j'-1}$. Then $X/Y$ is a Frattini chief factor of $L$. Furthermore the following conditions are satisfied.
\begin{itemize}
\item[(i)] If $A\cap Y=B\cap Y$, then $A\cap Y=B\cap X$ and 
\[ \frac{A}{B} \searrow \frac{A\cap X}{B\cap Y} \swarrow \frac{X}{Y}. 
\]
Moreover, $A+Y=A+X=B+X$ and 
\[ \frac{A}{B} \swarrow \frac{A+X}{B+Y} \searrow \frac{X}{Y}.
\]
\item[(ii)] If $A\cap Y \neq B\cap Y$ then 
\[ \left[\frac{A\cap X}{B\cap X} \searrow \frac{A\cap Y}{B\cap Y} \right]
\]
 is a crossing and 
\[\frac{A}{B} \searrow \frac{A\cap X}{B\cap X} \hbox{ and } \frac{A \cap X}{A \cap Y} \swarrow \frac{X}{Y}. 
\]
\end{itemize}
In particular, in both cases, $(A\cap X)/(B \cap Y)$ and $(A \cap X)/(B\cap X)$ are Frattini chief factors of $L$.
\end{propo}
\begin{proof} This is simply the dual of Proposition \ref{p:1}.
\end{proof}
\bigskip

We say that two chief factors $A/B$ and $C/D$ of $L$ are {\em m-related} if one of the following holds.
\begin{itemize}
\item[1.] There is a supplemented chief factor $R/S$ such that $A/B \swarrow R/S \searrow C/D$.
\item[2.] There is an $m$-crossing $[U/V \searrow W/X]$ such that $A/B \swarrow V/X$ and $W/X \searrow C/D$.
\item[3.] There is a Frattini chief factor $Y/Z$ such that $A/B \searrow Y/Z \swarrow C/D$.
\item[4.] There is an $m$-crossing $[U/V \searrow W/X]$ such that $A/B \searrow U/V$ and $U/W \swarrow C/D$.
\end{itemize}

\begin{theor}\label{t:2} Suppose that $A/B$ and $C/D$ are $m$-related chief factors of $L$. Then
\begin{itemize}
\item[(i)] $A/B$ and $C/D$ are $L$-connected;
\item[(ii)] $A/B$ is Frattini if and only if $C/D$ is Frattini; and
\item[(iii)] if $A/B$ and $C/D$ are supplemented, then there exists a common supplement.
\end{itemize}
\end{theor}
\begin{proof} 
\begin{itemize}
\item[(i)] In case $1$ we have 
\[ \frac{A}{B}=\frac{A}{A\cap S} \cong_L \frac{A+S}{S}= \frac{C+S}{S} \cong_L \frac{C}{C\cap S}=\frac{C}{D}.
\]
In case $3$ we have 
\[ \frac{A}{B}=\frac{B+Y}{B} \cong_L \frac{Y}{B\cap Y}= \frac{Y}{Z}= \frac{Y}{D\cap Y} \cong_L \frac{D+Y}{Y}=\frac{C}{D}.
\]
Consider case $2$. Here $V/X$ and $W/X$ have a common supplement, $M$ say, by Theorem \ref{t:1}. Then $(V+M_L)/M_L$ and $(W+M_L)/M_L$ are minimal ideals of the primitive Lie algebra $L/M_L$. If $V+M_L=W+M_L$ then $V/X \cong_L W/X$, which implies that $A/B \cong_L C/D$. Otherwise $L/M_L$ is a primitive Lie algebra of type $3$ whose minimal ideals are $(V+M_L)/M_L$ and $(W+M_L)/M_L$. Since $A/B \cong_L V/X$ and $C/D \cong_L W/X$ we see that $A/B$ and $C/D$ are $L$-connected.
\par

Case $4$ is similar to case $2$.
\item[(ii)] If $A/B$ is Frattini, then case $1$ of the definition of `$m$-related' cannot hold. Suppose we are in case $2$. Then $[U/W \searrow V/X]$ is an $m$-crossing, by Theorem \ref{t:1}, so $V/X$ is supplemented in $L$. Hence $A/B$ is supplemented in $L$, by Lemma \ref{l:1}, so case $2$ cannot hold. If case $3$ holds, then $C/D$ is Frattini, by Lemma \ref{l:1} (i). In case $4$, $[U/W \searrow V/X]$ is an $m$-crossing, by Theorem \ref{t:1}. But then $U/W$ is Frattini, whence $C/D$ is Frattini, by Lemma \ref{l:1}. 
\item[(iii)] Let $A/B$ and $C/D$ be supplemented. Then we are in either case $1$ or case $2$ of the definition of `$m$-related'. In case $1$, if $M$ supplements $R/S$ then $M$ supplements both $A/B$ and $C/D$, by Lemma \ref{l:1}. So suppose that case $2$ holds, Then there is a common supplement $M$ to $V/X$ and $W/X$, by Theorem \ref{t:1}. But $M$ also supplements $A/B$ and $C/D$, by Lemma \ref{l:1}
\end{itemize}
\end{proof}

\begin{lemma}\label{l:4} With the same hypotheses as in Lemma \ref{l:3} suppose that $A/B$ and $Y_j/Y_{j-1}$ are $m$-related. Then
\begin{itemize}
\item[(i)] $A/B$ and $Y_j/Y_{j-1}$ are supplemented in $L$ if and only if $(A+Y_{j-1})/(B+Y_{j-1})$ is supplemented in $L$; and
\item[(ii)] $A/B$ and $Y_j/Y_{j-1}$ are Frattini in $L$ if and only if $A\cap Y_j/B\cap Y_j$ is Frattini in $L$.
\end{itemize}
\end{lemma}
\begin{proof}
\begin{itemize}
\item[(i)] Put $C=Y_j$, $D=Y_{j-1}$  and suppose that case $1$ of the definition of `$m$-related' holds. If $A+D=B+D$ then $R=A+D+S=B+D+S=S$, a contradiction, so $B+D \subset A+D$. It follows from Lemma \ref{l:3} (iii) that $(A+D)/(B+D) \searrow A/B$, and, in particular, that $(A+D)/(B+D)$ is a chief factor of $L$. But $B+D \subseteq (A+D)\cap S \subseteq A+D$. Since $A+D \not \subseteq S$ we have that $(A+D)\cap S=B+D$ and $R/S \searrow (A+D)/(B+D)$. Hence $(A+Y_{j-1})/(B+Y_{j-1})$ is supplemented in $L$.
\par

Now suppose that case $2$ holds. If $A+D=B+D$, then $V=A+X=A+D+X=B+D+X=X$, a contradiction, so $B+D \subset A+D$ and, as above, $(A+D)/(B+D) \searrow A/B$. Now $V=A+D+X$ and $(A+D)\cap X=A\cap X+D=B+D$, so $V/X \searrow (A+D)/(B+D)$. Hence $(A+Y_{j-1})/(B+Y_{j-1})$ is supplemented in $L$.
\par

The converse follows from Lemma \ref{l:3} (iii).
\item[(ii)] This the dual statement to (i).
\end{itemize}
\end{proof}

\section{A generalised Jordan-H\"{o}lder Theorem}

\begin{theor}\label{t:3}  Let $K$ and $H$ be ideals of $L$ such that $K \subset H$ and two sections of chief series of $L$ between $K$ and $H$ are
\[ K=X_0 \subset X_1 \subset \ldots \subset X_n=H
\] and
\[ K=Y_0 \subset Y_1 \subset \ldots \subset Y_m=H.
\]
Then $n=m$ and there is a unique permutation $\sigma \in S_n$ such that $X_i/X_{i-1}$ and $Y_{\sigma(i)}/Y_{\sigma(i)-1}$ are $m$-related, for $1 \leq i \leq n$. Furthermore
\[ \sigma(i)=max \{ j : (X_i+Y_{j-1})/(X_{i-1}+Y_{j-1}) \hbox{ is supplemented in } L \}
\] if $X_i/X_{i-1}$ is supplemented in $L$, and
\[ \sigma(i)=min \{ j : X_i\cap Y_j/X_{i-1}\cap Y_j \hbox{ is Frattini in } L \}
\] if $X_i/X_{i-1}$ is Frattini in $L$
\end{theor}
\begin{proof} We can assume without loss of generality that $n \geq m$. Put $A=X_i$, $B=X_{i-1}$, $X=Y_{\sigma(i)}$, $Y=Y_{\sigma(i)-1}$.
\par

By Proposition \ref{p:1}, if $A/B$ is supplemented in $L$, then so is $X/Y$. Moreover, if $A+X=B+X$ then $A/B \swarrow (A+Y)/(B+Y) \searrow X/Y$, by Proposition \ref{p:1} (i). Also $(A+Y)/(B+Y)$ is supplemented in $L$, by the definition of $\sigma(i)$. Thus, this is case $1$ of the definition of `$m$-related'. If $A+X \neq B+X$ then we are in case $2$ of the definition, by Proposition \ref{p:1} (ii). 
\par

Dually, by Proposition \ref{p:2}, if $A/B$ is Frattini, then so is $X/Y$. Moreover, if $A\cap X=B\cap X$, then $A/B \searrow A\cap X/B\cap Y \swarrow X/Y$, by Proposition \ref{p:2} (i). Also $A\cap X/A\cap Y$ is Frattini, by the definition of $\sigma(i)$. Thus, this is case $3$ of the definition of `$m$-related'. If $A\cap X \neq B\cap X$ then we are in case $4$ of the definition, by Proposition \ref{p:2} (ii).
\par

Therefore, in all cases, $A/B$ and $X/Y$ are $m$-related for $1 \leq i \leq n$.
\par

Next we show that the map $\sigma : \{1, \ldots, n\} \rightarrow \{1, \ldots, m\}$ defined in the statement of the theorem is injective. Put $C=X_k$ and $D=X_{k-1}$, where $i < k$ and $\sigma(i)=\sigma(k)$.
\par

Suppose that $A/B$ is supplemented in $L$; then so are $X/Y$ and $C/D$. Suppose that $A+X=B+X$. Then $A \subseteq D$, so $A+X=A+Y$, by Proposition \ref{p:1} (i),  which yields that $D+X=D+A+X=D+A+Y=D+Y$. Since $C/D$ is supplemented in $L$ and $\sigma(k)=j$, $(C+Y)/(D+Y)$ is a chief factor of $L$, and then $D+X=D+Y \subset C+Y=C+X$. It follows from Proposition \ref{p:1} (ii) that $(D+X)/(D+Y) \searrow X/Y$; in particular, $D+X \neq D+Y$, a contradiction.
\par

Hence $B+X \subset A+X$. Then $[(A+X)/(B+X) \searrow (A+Y)/(B+Y)]$ is an $m$-crossing, by Proposition \ref{p:1} (ii), and hence so is $[(A+X)/(A+Y) \searrow (B+X)/(B+Y)]$, by Theorem \ref{t:1}. It follows that the chief factor $(A+X)/(A+Y)$ is Frattini. Since $\sigma(k)=j$ we have that $(C+Y)/(D+Y)$ and $(D+X)/(D+Y)$ are supplemented chief factors of $L$. But $A \subseteq D$ so $A+Y \subseteq D+Y$ and $A+X \subseteq D+X$. Also $D+X=(D+Y)+(A+X)$ and $(D+Y)\cap (A+X)=A+D\cap X+Y$. But $Y \subseteq D\cap X+Y \subseteq X$ and $X/Y$ is a chief factor of $L$. If $D\cap X +Y=X$ then $X \subseteq D+Y$ and  so $D+X=D+Y$, contradicting the fact that $(D+X)/(D+Y)$ is a chief factor of $L$. Hence $D\cap X+Y=Y$, giving $D\cap X \subseteq Y$ and $(D+Y)\cap (A+X)=A+Y$. Thus $(D+X)/(D+Y) \searrow (A+X)/(A+Y)$, which implies that $(D+X)/(D+Y)$ is Frattini, by Lemma \ref{l:1}, which is a contradiction
\par

We have shown that the restriction of $\sigma$ to the subset ${\mathcal I}$ of $\{1, \ldots, n \}$ composed of all indices $i$ corresponding to the supplemented chief factors $X_i/X_{i-1}$ is injective. Applying dual arguments shows that the restriction of $\sigma$ to the subset of $\{1, \ldots, n\} \setminus {\mathcal I}$ consisting of all Frattini chief factors $X_i/X_{i-1}$ is injective. By arguments at the beginning of the proof, $\sigma$ is injective. Hence $n=m$ and $\sigma \in S_n$.
\par

Finally, if $\tau$ is any permutation with the above properties then the definition of $\sigma$ requires $\tau(i)=\sigma(i)$ for all $i \in {\mathcal I}$ and $\tau(i)=\sigma(i)$ for all $i \in \{1, \ldots, n\} \setminus {\mathcal I}$, by Lemma \ref{l:4}. Hence $\tau = \sigma$.
\end{proof}

\begin{coro} Let $\sigma$ be the permutation constructed in Theorem \ref{t:3}. If $X_i/X_{i-1}$ and $Y_{\sigma(i)}/Y_{\sigma(i)-1}$ are supplemented, then they have a common supplement. Moreover, the same is true if we replace `supplement' by `complement'. 
\end{coro}
\begin{proof} The first assertion follows immediately from Theorem \ref{t:2}. The second is clear if both chief factors are abelian. So suppose that they are complemented nonabelian chief factors. Then case $2$ of the definition of $m$-related cannot hold, since $W/X$ (and thus $Y_{\sigma(i)}/Y_{\sigma(i)-1}$) would have to be abelian, by the remark immediately preceding Theorem \ref{t:1}. Case $3$ cannot hold, since $Y/Z$ is not supplemented, and case $4$ cannot hold, since $U/W$ is not supplemented. Hence case $1$ holds, and there is a supplemented chief factor $R/S$ such that $X_i/X_{i-1} \swarrow R/S \searrow Y_{\sigma(i)}/Y_{\sigma(i)-1}$, so $R=X_i+S=Y_{\sigma(i)}+S$, $X_i\cap S=X_{i-1}$ and $Y_{\sigma(i)}\cap S=Y_{\sigma(i)-1}$. 
\par

Let $M$ be a complement of $X_i/X_{i-1}$, so $L=X_i+M$ and $X_i\cap M=X_{i-1}$. Moreover, it is also a supplement of $R/S$ and $Y_{\sigma(i)}/Y_{\sigma(i)-1}$, by Lemma \ref{l:1} (iv), so $L=R+M=Y_{\sigma(i)}+M$, $S \subseteq R\cap M$ and $Y_{\sigma(i)-1}\subseteq Y_{\sigma(i)}\cap M$.  Then \[R+M_L=X_i+S+M_L=X_i+M_L=Y_{\sigma(i)}+M_L
\] and 
\[M\cap (Y_{\sigma(i)}+M_L)=M\cap(X_i+M_L)=X_{i-1}+M_L=M_L. 
\]Hence $M\cap Y_{\sigma(i)}=M_L\cap Y_{\sigma(i)}=Y_{\sigma(i)-1}$ and $M$ complements $Y_{\sigma(i)}/Y_{\sigma(i)-1}$.
\end{proof}

\end{document}